\documentclass[12pt]{amsart}
\usepackage{amsfonts}
\usepackage{amssymb}
\usepackage{bm}

 \newtheorem{theorem}{Theorem}[section]
 \newtheorem{cor}[theorem]{Corollary}
 
 \newtheorem{lemma}[theorem]{Lemma}
 
 \theoremstyle{definition}

 \newtheorem{remark}[theorem]{Remark}

\numberwithin{equation}{section}
\DeclareMathOperator{\supp}{supp}

\begin{document}

\title[additive bases of abelian groups of rank 2]
{additive bases of abelian groups of rank 2}

\author[W.D. Gao]{Weidong Gao}
\address{Center for Applied Mathematics \\ Tianjin University\\
Tianjin, 300072, P.R. China}
\email{wdgao@nankai.edu.cn, weidong.gao@tju.edu.cn, wdgao1963@aliyun.com}

\author[Y.L. Li]{Yuanlin Li}
\address{Department of Mathematics and Statistics\\  Brock University\\ St. Catharines, ON L2S 3A1, Canada}
\email{yli@brocku.ca}

\author[Y.K. Qu]{Yongke Qu}
\address{Department of Mathematics\\ Luoyang Normal University\\
LuoYang 471934, P.R. China}
\email{yongke1239@163.com}

\author[Q.H. Wang]{Qinghong Wang}
\address{College of Science  \\ Tianjin University of Technology \\ Tianjin, 300384, P.R. China}
\email{wqh1208@aliyun.com}


\begin{abstract}
Let $G$ be a finite abelian group and $p$ be the smallest prime dividing $|G|$. Let $S$ be a sequence over $G$.
We say that $S$ is regular if for every proper subgroup $H \subsetneq G$, $S$ contains at most $|H|-1$ terms from
$H$. Let $\mathsf c_0(G)$ be the smallest integer $t$ such that every regular sequence $S$ over $G$ of length
$|S|\geq t$ forms an additive basis of $G$, i.e., $\sum(S)=G$. The invariant $\mathsf c_0(G)$ was first studied
by Olson and Peng in 1980's, and since then it has been determined for all finite abelian groups except for the
groups with rank 2 and a few groups of rank 3 or 4 with order less than $10^8$. In this paper, we focus on the
remaining case concerning groups of rank 2. It was conjectured by Gao et al. (Acta Arith. 168 (2015) 247-267) that $\mathsf c_0(G)=m(G)$. We confirm the conjecture for the case when $G=C_{n_1}\oplus C_{n_2}$ with $n_1|n_2$, $n_1\geq 2p$, $p\geq 3$ and $n_1n_2\geq 72p^6$.
\end{abstract}

\date{}

\maketitle

\section{Introduction and main results}
Through the paper, let $G$ be a finite abelian group, written additively, $p$ be the smallest prime dividing $|G|$ and $\mathsf r(G)$ denote the rank of $G$. Let $S$ be a sequence over $G$. We say that $S$ is an {\sl additive basis} of $G$ if every element of $G$ can be expressed as the sum over a nonempty subsequence of $S$. For every subgroup $H$ of $G$, let $S_H$ denote the  subsequence of $S$ consisting of all terms of $S$ contained in $H$. We say that $S$ is a regular sequence over $G$ if $|S_H|\leq |H|-1$ holds for every proper subgroup $H \subsetneq G$. Let $\mathsf c_0(G)$ denote the smallest integer $t$ such that every regular sequence over $G$ of length at least $t$ is an additive basis of $G$. The problem of determining $\mathsf c_0(G)$ was first proposed by Olson and it was conjectured that $\mathsf c_0(C_p\oplus C_p)=2p-1$. In 1987, Peng proved this conjecture and further determined $\mathsf c_0(G)$ for all the finite elementary abelian $p$-groups (\cite{Peng1,Peng2}). Those results and techniques used for determining $\mathsf c_0(G)$ have been successfully applied to the study of nonunique factorization (\cite{GPZ2013}).  Recently, the problem related to the additive basis of a finite abelian group has been investigated by several authors (\cite{GHQQZ, GQZ,QH1,QH2,QL}). Let
\[
 m(G)= \left\{ \begin{array}{ll} & |G|, \mbox{ if } G \mbox{ is cyclic}, \\
                                 & 2p-1, \mbox{ if } G =C_p\oplus C_p, \\
                                 & kp+2p-3, \mbox{ if } G=C_p\oplus C_{pk}\mbox{ and } k\geq 2, \\
                                 & \frac{|G|}{p}+p-2, \mbox{ otherwise}.
\end{array} \right.
\]
It has been proved that $\mathsf c_0(G)=m(G)$ for any of the following finite abelian groups:
\begin{enumerate}
\item $G$ is cyclic;

\item $|G|$ is even;

\item $\mathsf r(G)\geq 4$ and $G\neq C_3^3\oplus C_{3n}$ where $n>3$ is odd and is not a power of $3$ with $|G|<
    3.72\times 10^7$ ;

\item $\mathsf r(G)=3$ and either $p\geq 11$ or $3\leq p \leq 7$ with $|G| \geq 3.72\times 10^7$;

\item $\mathsf r(G)\geq 2$ and $G$ is a $p$-group;

\item $G=C_3\oplus C_{3q}$, where $q$ is a prime.
\end{enumerate}

In this paper, we focus our investigation on the remaining case when $G$ is of rank $2$. It was conjectured in \cite{GHQQZ} that $\mathsf c_0(G)=m(G)$. We confirm this conjecture for the case when $G=C_{n_1}\oplus C_{n_2}$ with $n_1|n_2$, $n_1\geq 2p$, $p\geq 3$ and $n_1n_2\geq 72p^6$.

\begin{theorem}\label{mainthm}
Let $G=C_{n_1}\oplus C_{n_2}$ with $n_1|n_2$, and $p\geq 3$ be the smallest prime divisor of $|G|$. If $n_1\geq 2p$ and $|G|\geq 72p^6$, then $\mathsf c_0(G)=|G|/p+p-2$.
\end{theorem}

\section{Notations and Preliminaries}
Suppose that $G_0\subseteq G$ is a subset of $G$ and $\mathcal{F}(G_0)$ is the multiplicatively written, free
abelian monoid with basis $G_0$. The elements of $\mathcal{F}(G_0)$ are called {\it sequences} over $G_0$. We denote multiplication in $\mathcal{F}(G_0)$ by the bold symbol $\bm\cdot$ rather than by juxtaposition, and use brackets for all exponentiation in $\mathcal{F}(G_0)$.

A sequence $S \in \mathcal F(G)$ will be written in the form $S= g_1 \bm \cdot \ldots \bm\cdot g_{\ell},$ where $|S|= \ell$ is the {\it length} of $S$. For $g \in G$, let $\mathsf v_g(S) = |\{ i\in [1, \ell] : g_i =g \}|\,  $ denote the {\it multiplicity} of $g$ in $S$. We call $\mbox{supp}(S)=\{g\in G: \mathsf v_g(S)>0\}$ the {\it support} of $S$. Let $\mathsf h(S) = \mbox{max}\{\mathsf v_g(S): g\in G\}$. A sequence $T \in \mathcal F(G)$ is called a {\it subsequence} of $S$ and is denoted by $T \mid S$ if  $\mathsf v_g(T) \le \mathsf v_g(S)$ for all $g\in G$. Denote by $S \bm\cdot T^ {[-1]}$  the subsequence of $S$ obtained by removing the terms of $T$ from $S$.

If $S_1, S_2 \in \mathcal F(G)$, then $S_1 \bm\cdot S_2 \in \mathcal F(G)$ denotes the sequence satisfying that $\mathsf v_g(S_1 \bm\cdot S_2) = \mathsf v_g(S_1 ) + \mathsf v_g( S_2)$ for all $g \in G$. For convenience we  write
\begin{center}
 $g^{[k]} = \underbrace{g \bm\cdot \ldots \bm\cdot g}_{k} \in \mathcal F(G)\quad$
\end{center}
for $g \in G$ and $k$ a nonnegative integer.

Let $\sigma(S)=\sum_{g\in G}\mathsf v_g(S)g\in G$ be the sum of $S$. Define
  $$
  \sum(S)=\{\sigma(T): \ 1 \neq  T\mid S\},
  $$
\noindent where $1$ is the empty sequence, and $$
  \sum\nolimits_0(S)=\sum(S)\cup\{0\}.
  $$
\noindent We call a sequence $S$ a {\it zero-sum} sequence if $\sigma(S)=0$, and a {\it zero-sum free} sequence if $0\notin
\sum(S)$. We note that a subset of $G$ can be regarded as a sequence, so all above-mentioned concepts and notations  regarding a sequence of $G$ are valid for a subset of $G$. We say that a subset $A\subset G\setminus \{0\}$ is {\it $2$-zero-sum free} if $A$ contains no two distinct elements with sum zero.

Let $\mathsf D(G)$ denote the Davenport constant of $G$, which is defined as the smallest integer $t$ such that
every sequence $S$ over $G$ of length $|S|\geq t$ contains a nonempty zero-sum subsequence. Note that
$\mathsf D(C_{n_1}\oplus C_{n_2})=n_1+n_2-1$, where $1\leq n_1|n_2$ (\cite[Theorem 5.8.3]{GeK}).

For each subset $A$ of $G$, denote by $\langle A \rangle$ the subgroup generated by $A$.  Let ${\rm st}(A)=\{g\in
G: g+A=A\}$. Then ${\rm st}(A)$ is the maximal subgroup $H$ of $G$ such that $H+A=A$. The following is the well
known Kneser's theorem and a proof of it can be found in \cite{Na}.

\begin{lemma}(Kneser) \cite[Theorem 4.4]{Na}\label{Kneser} Let $A_1, \ldots, A_r$ be nonempty finite subsets of an
abelian group $G$, and let $H={\rm st}(A_1+\cdots +A_r)$. Then,
$$
|A_1+\cdots +A_r|\geq |A_1+H|+\cdots +|A_r+H|-(r-1)|H|.
$$
\end{lemma}

\begin{lemma}\cite[Lemma 2.2]{GQZ}\label{lowerbound}
$\mathsf c_0(G)\geq m(G)$ for every finite abelian group $G$.
\end{lemma}

\begin{lemma}\cite[Lemma 2.3]{GQZ}\label{stabilizer}
Let $S$ be a regular sequence over $G$ of length $|S|\geq \mbox{max}\{|G|/p+p-2, \mathsf D(G)\}$. Let $T$ be a nonempty subsequence of $S$. If $\sum(S)\neq G$, then
\begin{enumerate}
\item ${\rm st}(\sum_0(T))=\{0\}$;
\item $|\sum_0(T))|\geq |T|+1$.
\end{enumerate}
\end{lemma}

We remark  that since $|S|\geq \mathsf D(G)$, we have $0\in \sum(S)$. Thus $\sum(S)=\sum_0(S)$, and the above lemma follows immediately from \cite[Lemma 2.3]{GQZ}.

For a finite abelian group $G$ and a positive integer $k$, let $f(G,k)=\mbox{min}\{|\sum(S)|:S \mbox{ is a subset of } G \mbox{ with } |S|=k \mbox{ and } 0\notin \sum(S)\}$. If $0\in \sum(S)$ for every $S\subseteq G$ of $|S|=k$, we set $f(G,k)=\infty$. Let $f(k)=\mbox{min}\{f(G,k): G \mbox{ is a finite abelian group}\}$.

\begin{lemma}\cite[Theorem 1.1]{GHHLLP}\label{f6}
$f(k)\geq \frac{k^2}{6}$ holds for every positive integer $k$.
\end{lemma}

\section{Proof of Theorem \ref{mainthm}}
We first prove a few technical results.
\begin{lemma}\label{subsetsums}
Let $G$ be an abelian group and $A\subseteq G\setminus\{0\}$. If $|A|\geq 6p(p+1)+1$, then there is a zero-sum free subset $B\subseteq A$ such that $|\sum_0(B)|\geq p|B|+2$ with $|B|\leq 6p+1$.
\end{lemma}
\begin{proof} It is sufficient to prove that there is a zero-sum free subset $B\subseteq A$ such that $|\sum(B)|\geq p|B|+1$ with $|B|\leq 6p+1$. We first prove that for every $k\leq 6p+1$ there is  a zero-sum free subset $B\subseteq A$ such that either $|B|=k$ or $|B|< k$ and $|\sum(B)|\geq p|B|+1$.

 We proceed by induction on $k$. If $k=1$, then the result holds trivially. Assume that the result is true for each $k$ ($k<6p+1$). We next prove it is also true for $k+1$. By the induction hypothesis, there is  a zero-sum free subset $B_0\subseteq A$ such that either $|B_0|=k$ or $|B_0|< k$ and $|\sum(B_0)|\geq p|B_0|+1$. If $|\sum(B_0)|\geq p|B_0|+1$, then let $B=B_0$. Hence $|B|<k+1$ and $|\sum(B)|\geq p|B|+1$, so we are done. Thus we may assume that $|B_0|=k$ and $|\sum(B_0)|\leq p|B_0|$. It follows that
 $|-\sum(B_0)|=|\sum(B_0)|\leq pk\leq 6p^2\leq |A|-(6p+1)<|A|-|B_0|=|A \setminus B_0|$. Hence, there is an element $g\in A \setminus B_0$ such that $g \not \in -\sum(B_0)$. So, $B_0 \cup \{g\}$ is a zero-sum free subset of $A$. Set $B=B_0\cup \{g\}$, and then the result holds for $k+1$.

 Next, let $k=6p+1$. We have just proved that there is a zero-sum free subset $B$ of $A$ such that either $|B|=6p+1$ or $|\sum(B)|\geq p|B|+1$ and $|B|<6p+1$. If the latter is true, then we are done. So, we may assume that $|B|=6p+1$.  By Lemma~\ref{f6}, $|\sum(B)|\geq f(6p+1)\geq (6p+1)^2/6\geq p(6p+1)+1$ and again we are done.
\end{proof}

\begin{lemma}\label{T}
Let $G$ be a finite abelian group with smallest prime divisor $p$ of $|G|$. Let $T$ be a sequence over $G\setminus \{0\}$ and $H=\langle \mbox{supp}(T)\rangle$. Then $|\sum_0(T)|\geq |T|+1$ if one of the following holds
\begin{enumerate}
\item $|T|\leq p-1$;
\item $|T|\leq 2p-1$ and $|H|\geq 2p$.
\end{enumerate}
\end{lemma}
\begin{proof}
 Let $T=g_1\bm\cdot\ldots \bm\cdot g_{k}$. Then,
 $$
 \sum\nolimits_0(T)=\{{g_1},0\}+\cdots+\{{g_{k}},0\}.
 $$
Let $M=\mbox{st}(\{{g_1},0\}+\cdots+\{{g_{k}},0\})$. If $M=\{0\}$, then by Lemma~\ref{Kneser}, $|\sum_0(T)|=|\{{g_1},0\}+\cdots+\{{g_{k}},0\}|\geq k+1=|T|+1$ as desired.

Next we assume $M \neq \{0\}$, so $|M|\geq p$.

(1) Given that $|T|\leq p-1$. Clearly, $|\sum_0(T)|=|\{{g_1},0\}+\cdots+\{{g_{k}},0\}|=|\{{g_1},0\}+\cdots+\{{g_{k}},0\}+M|\geq|M|\geq k+1$ as desired.

(2) Given that $|T|\leq 2p-1$ and $|H|\geq 2p$. If $M\supseteq H$, then as above $|\sum_0(T)|\geq|M|\geq|H|\geq 2p\geq |T|+1$. Next we assume that $\{0\} \neq M \nsupseteq H$. Clearly, $\sum_0(T)=\{{g_1},0\}+\cdots+\{{g_{k}},0\}\subsetneq M$. Thus $g_i\notin M$ for some $1\leq i\leq k$, and so $|\{{g_i},0\}+M|=2|M|$. By Lemma~\ref{Kneser}, $|\sum_0(T)|=|\{{g_1},0\}+\cdots+\{{g_{k}},0\}|\geq 2|M|\geq 2p\geq |T|+1$ as desired.
\end{proof}
As a consequence, we  obtain the following corollary, which will be used in the proof for our main result.
\begin{cor}\label{coset}
Let $G$ be a finite abelian group with smallest prime divisor $p$ of $|G|$. Let $S$ and  $T$ be sequences over $G$. Suppose $g\notin \langle \mbox{supp}(S)\rangle$ for every $g|T$. Then $|\sum_0(T\bm\cdot S)|\geq (|T|+1)|\sum_0(S)|$ if one of the following holds
\begin{enumerate}
\item $|T|\leq p-1$;
\item $|T|\leq 2p-1$ and $|\langle \mbox{supp}(T\bm\cdot S)\rangle/\langle \mbox{supp}(S)\rangle|\geq 2p$.
\end{enumerate}
\end{cor}
\begin{proof} Let $\varphi: \langle \mbox{supp}(T\bm\cdot S)\rangle\rightarrow \langle \mbox{supp}(T\bm\cdot S)\rangle/\langle \mbox{supp}(S)\rangle$ be the canonical epimorphism. Let $T=g_1\bm\cdot\ldots \bm\cdot g_{k}$ and $\varphi(T)=\varphi(g_1)\bm\cdot\ldots \bm\cdot \varphi(g_{k})$. Since $g_i\notin \langle \mbox{supp}(S)\rangle$ for every $i\in[1,k]$, we have $\varphi(g_i)\neq 0$ for $i\in [1,k]$. Thus $\varphi(T)$ is a sequence over $\varphi(G)\setminus \{ \varphi(0)\}$. By Lemma~\ref{T}, we have $|\sum_0(\varphi(T))|\geq |\varphi(T)|+1=|T|+1$. Therefore, $|\sum_0(T\bm\cdot S)|\geq |\sum_0(\varphi(T))|\cdot|\sum_0(S)|\geq (|T|+1)|\sum_0(S)|$.
\end{proof}

\begin{lemma}\cite[Lemma 3.1]{GHQQZ}\label{A=2}
If $A$ is a $2$-zero-sum free subset of three elements in an abelian group, then either $|\sum_0(A)|\geq 7$ or $A$ contains some element of order two.
\end{lemma}

We now prove our main result.
\\

\noindent {\bf Proof of Theorem \ref{mainthm}.}
\\

By Lemma~\ref{lowerbound}, it suffices to prove $\mathsf c_0(G)\leq n/p+p-2$. Let $S$ be a regular sequence over $G$ of length $|S|=n/p+p-2$. We only need to show that $\sum(S)=G$. Since $2p\leq n_1|n_2$, we have $n/p+p-2\geq n_1+n_2-1=\mathsf D(G)$. Hence $0\in \sum(S)$ and thus $\sum(S)=\sum_0(S)$.

Assume to the contrary that $\sum\nolimits_0(S)\neq G$, that is $|\sum\nolimits_0(S)|<n$. Thus, by Lemma~\ref{stabilizer}~(1), we have ${\rm st}(\sum\nolimits_0(S))=\{0\}$.

Suppose that $\mathsf h(S)\leq 2p-2$. If $|\supp(S)|\leq 6p(p+1)$, then $|S|\leq |\supp(S)|\cdot \mathsf h(S)\leq 12p(p^2-1)$. Since $|S|=\frac{n}{p}+p-2\geq \frac{72p^6}{p}+p-2\geq 72p^5>12p(p^2-1)$, we get a contradiction. Hence $|\supp(S)|\geq 6p(p+1)+1$. By Lemma~\ref{subsetsums}, we may choose $A_1\subseteq \supp(S)\subseteq G$, such that $|A_1|\leq 6p+1$ and $|\sum_0(A_1)|\geq p|A_1|+2$. Let $t\geq 1$ be the maximal integer such that $S$ has a factorization $$S=A_1\bm\cdot\ldots\bm\cdot A_t\bm\cdot T$$ with $A_i\subseteq G$, $|A_i|\leq 6p+1$ and $|\sum_0(A_i)|\geq p|A_i|+2$ for every $i\in [1,t]$.

By the maximality of $t$ and Lemma~\ref{subsetsums}, we have $|\mbox{supp}(T)|\leq 6p(p+1)$. Thus $|T|\leq |\mbox{supp}(T)|\cdot\mathsf h(S)\leq 12p(p^2-1)$. Note that $$\sum(S)=\sum\nolimits_0(S)=\sum\nolimits_0(A_1)+\cdots+\sum\nolimits_0(A_t)+\sum\nolimits_0(T).$$ Since $|A_i|\leq 6p+1$ for each $i\in[1,t]$, we have $t\geq \frac{|S|-|T|}{6p+1}$. By Lemma~\ref{stabilizer}~(2), $|\sum\nolimits_0(T)|\geq |T|+1$. Since ${\rm st}(\sum(S))=\{0\}$, by Lemma~\ref{Kneser}, we have
\begin{align}
|\sum\nolimits_0(S)|=&|\sum\nolimits_0(A_1)+\cdots+\sum\nolimits_0(A_t)+\sum\nolimits_0(T)|\notag \\
\geq &(|\sum\nolimits_0(A_1)|-1)+\cdots+(|\sum\nolimits_0(A_t)|-1)+|\sum\nolimits_0(T)|\notag \\
\geq &\sum_{i=1}^t(p|A_i|+1)+|T|+1\notag \\
=&p(|S|-|T|)+t+|T|+1\notag \\
\label{calculation}\geq &(p+\frac{1}{6p+1})(|S|-|T|)+|T|+1\\
\geq &n+\frac{n}{p(6p+1)}+p^2-2p-(p-\frac{6p}{6p+1})|T|+1\  (\mbox{as }|S|=\frac{n}{p}+p-2)\notag \\
\geq &n+\frac{n}{p(6p+1)}+p^2-2p-\frac{12p^2(6p-5)(p^2-1)}{6p+1}+1\  (\mbox{as }|T|\leq 12p(p^2-1))\notag \\
\geq &n+\frac{n}{p(6p+1)}-\frac{12p^2\cdot6p\cdot p^2}{6p+1}\notag \\
\geq &n   \ \ \ \ \ \ \ \ \ (\mbox{as } n\geq 72p^6), \notag
\end{align}
yielding a contradiction. Therefore, we must have $\mathsf h(S)\geq 2p-1$. We choose $g_1\in \mbox{supp}(S)$ such that $$|S_{\langle g_1\rangle}|=\mbox{max}\{|S_{\langle g\rangle}|: g\in \mbox{supp}(S) \mbox{ and } \mathsf v_g(S)\geq 2p-1\}.$$ Let $S_1=S_{\langle g_1\rangle}$ and $\lambda$ be the maximal integer such that $S$ has a factorization $$S=(T_1\bm\cdot S_1)\bm\cdot (T_2\bm\cdot S_2)\bm\cdot \ldots \bm\cdot (T_{\lambda}\bm\cdot S_{\lambda})\bm\cdot S'$$ with the following properties:
\begin{itemize}
\item[(1)] $|T_i|=2p-1$ for each $i\in [1, \lambda]$;
\item[(2)] for each $i\in [1, \lambda]$, $\mathsf v_{g_i}(S_i)\geq 2p-1$ for some $g_i\in G$; moreover, $S_i$ contains all terms from $S\bm\cdot((T_1\bm\cdot S_1)\bm\cdot (T_2\bm\cdot S_2)\bm\cdot \ldots \bm\cdot (T_{i-1}\bm\cdot S_{i-1}))^{[-1]}$ contained in $\langle g_i\rangle$;
\item[(3)] $|\sum_0(T_i\bm\cdot S_i')|\geq 2p(|S_i'|+1)$ for all subsequences $S_i'|S_i$, where $i\in [1, \lambda]$.
\end{itemize}

Clearly, $S_1$ satisfies Property $(2)$.

\noindent {\bf Claim.}
$$\lambda\geq 1, \ \ \ \mbox{i.e., $T_1$ always exists}.$$
We first show that $\langle \mbox{supp}(S)\rangle=G$. For otherwise, if $H=\langle \mbox{supp}(S)\rangle<G$, then
$|H|\leq \frac{n}{p}$ where $p$ is the smallest prime divisor of $|G|$. Since $S$ is regular, we have $|H|-1\geq |S_H|=|S|=\frac{n}{p}+p-2\geq |H|$, yielding a contradiction. Thus $\langle \mbox{supp}(S)\rangle=G$. It follows that
$$
G/\langle \mbox{supp}(S_1)\rangle=\langle \mbox{supp}(S\bm\cdot S_1^{[-1]}), \mbox{supp}(S_1)\rangle/\langle \mbox{supp}(S_1)\rangle.
$$

Note that
$$
|S\bm\cdot S_1^{[-1]}|\geq \frac{n}{p}+p-2-\frac{n}{p^2}\geq 2p.
$$
Next we choose a suitable subsequence $T_1|S\bm\cdot S_1^{[-1]}$. If $G/\langle \mbox{supp}(S_1)\rangle$ is a cyclic group of prime order, then let $T_1$ be an arbitrary subsequence of $S\bm\cdot S_1^{[-1]}$ with length $|T_1|=2p-1$. Since $\langle\supp(T_1), \supp(S_1) \rangle/\langle \mbox{supp}(S_1)\rangle$ is not trivial, we must have \\
$\langle\supp(T_1), \supp(S_1) \rangle/\langle \mbox{supp}(S_1)\rangle=G/\langle \mbox{supp}(S_1)\rangle$. Therefore,
$$
|\langle\supp(T_1), \supp(S_1) \rangle/\langle \mbox{supp}(S_1)\rangle|=|G/\langle \mbox{supp}(S_1)\rangle|\geq n_1\geq 2p,
$$
as $|\langle \mbox{supp}(S_1)\rangle|\leq |\langle g_1\rangle|\leq n_2$. Next we assume that $|G/\langle \mbox{supp}(S_1)\rangle|$ is a composite number. Since $$G/\langle \mbox{supp}(S_1)\rangle=\langle \mbox{supp}(S\bm\cdot S_1^{[-1]}), \mbox{supp}(S_1)\rangle/\langle \mbox{supp}(S_1)\rangle,$$ there are two elements $h_1, h_2\in
\mbox{supp}(S\bm\cdot S_1^{[-1]})$ such that $|\langle h_1,h_2, \supp(S_1) \rangle /\langle \mbox{supp}(S_1)\rangle|$ is a composite number. It follows that
$$
|\langle h_1,h_2, \supp(S_1) \rangle /\langle \mbox{supp}(S_1)\rangle|\geq 2p,
$$
as $p$ is the smallest prime divisor of $n$. Let $T_1$ be a subsequence of $S\bm\cdot S_1^{[-1]}$ with $h_1\bm\cdot h_2|T_1$. Then,
$$
|\langle\supp(T_1), \supp(S_1) \rangle/\langle \mbox{supp}(S_1)\rangle|\geq 2p.
$$

So we can always choose a subsequence $T_1|S\bm\cdot S_1^{[-1]}$ such that \\
$|\langle\supp(T_1), \supp(S_1) \rangle/\langle \mbox{supp}(S_1)\rangle|\geq 2p$ and $|T_1|=2p-1$.

Finally, we verify Property $(3)$ holds. Let $S_1'|S_1$ be a subsequence of $S_1$. Then $\langle \mbox{supp}(S_1')\rangle\leq \langle\supp(S_1)\rangle$ and thus  \\
$|\langle\supp(T_1), \supp(S_1') \rangle/\langle \mbox{supp}(S_1')|\geq |\langle\supp(T_1), \supp(S_1) \rangle/\langle \mbox{supp}(S_1)\rangle|\geq 2p$.\\ Note that $g\notin \langle \mbox{supp}(S_1)\rangle$ for every $g|T_1$. Thus $g\notin \langle \mbox{supp}(S_1')\rangle$ for every $g|T_1$. By Corollary~\ref{coset}~(2) and Lemma~\ref{stabilizer}~(2), we have $|\sum_0(T_1\bm\cdot S_1')|\geq (|T_1|+1)|\sum_0(S_1')|\geq 2p(|S_1'|+1)$, implying Property $(3)$. In summary, we have found a subsequence $T_1$ (together with $S_1$) satisfying all three properties. This completes the proof of {\bf Claim}.

Next we distinguish the remaining proof into the following two cases according to the value of $\mathsf h(S')$.
\\

\noindent {\bf Case 1.} $\mathsf h(S')\leq 2p-2$.

Let $t\geq 0$ be the maximal integer such that $S'$ has a factorization $$S'=A_1\bm\cdot\ldots\bm\cdot A_t\bm\cdot T$$ with $A_i\subseteq G$, $|A_i|\leq 6p+1$ and $|\sum_0(A_i)|\geq p|A_i|+2$. By the maximality of $t$ and Lemma~\ref{subsetsums}, we have $|\mbox{supp}(T)|\leq 6p(p+1)$. Thus $|T|\leq |\mbox{supp}(T)|\cdot\mathsf h(S')\leq 12p(p^2-1)$. Since $|T_i\bm\cdot S_i|\geq 4p-2$ for each $i\in[1,\lambda]$, we have $\lambda\leq \frac{|S|-|S'|}{4p-2}$. Similarly, $t\geq \frac{|S'|-|T|}{6p+1}$. By Lemma~\ref{stabilizer}~(2), $|\sum\nolimits_0(T)|\geq |T|+1$. Since ${\rm st}(\sum\nolimits_0(S))=\{0\}$, by Lemma~\ref{Kneser}, we have
\begin{align*}
|\sum\nolimits_0(S)|=&|\sum_{i=1}^{\lambda}\sum\nolimits_0(T_i\bm\cdot S_i)+\sum_{j=1}^{t}\sum\nolimits_0(A_j)+ \sum\nolimits_0(T)|\\
\geq &\sum_{i=1}^{\lambda}(|\sum\nolimits_0(T_i\bm\cdot S_i)|-1)+\sum_{j=1}^{t}(|\sum\nolimits_0(A_j)|-1)+ |\sum\nolimits_0(T)|\\
\geq &\sum_{i=1}^{\lambda}(2p|S_i|+2p-1)+\sum_{j=1}^{t}(p|A_j|+1)+|T|+1\\
= &\sum_{i=1}^{\lambda}(2p|T_i\bm\cdot S_i|-(2p-1)^2)+p(|S'|-|T|)+t+|T|+1 \\
\geq &2p(|S|-|S'|)-(2p-1)^2\lambda+(p+\frac{1}{6p+1})(|S'|-|T|)+|T|+1\\
\geq &(p+1/2)(|S|-|S'|)+(p+\frac{1}{6p+1})(|S'|-|T|)+|T|+1\\
\geq &(p+\frac{1}{6p+1})(|S|-|T|)+|T|+1\\
\geq &n \ \ \ \ \ \ \ \ (\mbox{as } |T|\leq 12p(p^2-1)),
\end{align*}
yielding a contradiction. We note that the last inequality was obtained by using a similar calculation as in \eqref{calculation}.
\\

\noindent {\bf Case 2.} $\mathsf h(S')\geq 2p-1$.

There exists $g_{\lambda +1}\in G$ such that $\mathsf v_{g_{\lambda +1}}(S')=\mathsf h(S')\geq 2p-1$. Let $S_{\lambda+1}=S_{\langle g_{\lambda +1}\rangle}'$ and $U=S'\bm\cdot S_{\lambda+1}^{[-1]}$. Then $$S=(T_1\bm\cdot S_1)\bm\cdot (T_2\bm\cdot S_2)\bm\cdot \ldots \bm\cdot (T_{\lambda}\bm\cdot S_{\lambda})\bm\cdot S_{\lambda+1} \bm\cdot U.$$ We next distinguish the rest proof into the following two subcases.
\\

\noindent {\bf Subcase 2.1.} $|U|\leq 2p-2$.

By the choice of $S_1$, we have $|S_{\lambda+1}|\leq |S_1|$. Since $|T_i\bm\cdot S_i|\geq 4p-2$ for each $i\in[2,\lambda]$, we have $\lambda-1\leq \frac{|S|-|S_{1}\bm\cdot T_1|-|S_{\lambda+1}|-|U|}{4p-2}$. By Lemma~\ref{stabilizer}~(2), $|\sum\nolimits_0(S_{\lambda+1}\bm\cdot U)|\geq |S_{\lambda+1}\bm\cdot U|+1$. Since ${\rm st}(\sum_0(S))=\{0\}$, by Lemma~\ref{Kneser}, we have
\begin{align*}
|\sum\nolimits_0(S)|=&|\sum\nolimits_0(T_1\bm\cdot S_1)+\sum_{i=2}^{\lambda}\sum\nolimits_0(T_i\bm\cdot S_i)+ \sum\nolimits_0(S_{\lambda+1}\bm\cdot U)|\\
\geq &(|\sum\nolimits_0(T_1\bm\cdot S_1)|-1)+\sum_{i=2}^{\lambda}(|\sum\nolimits_0(T_i\bm\cdot S_i)|-1)+|\sum\nolimits_0(S_{\lambda+1}\bm\cdot U)|\\
\geq &(2p|S_1|+2p-1)+\sum_{i=2}^{\lambda}(2p|S_i|+2p-1)+|S_{\lambda+1}\bm\cdot U|+1\\
\geq &(p+1/2)|S_1|+\sum_{i=2}^{\lambda}(2p|T_i\bm\cdot S_i|-(2p-1)^2)+(p+1/2)|S_{\lambda+1}|+2p\\
= &(p+1/2)(|S_1|+|S_{\lambda+1}|)+2p(|S|-|T_1\bm\cdot S_{1}|-|S_{\lambda+1}|-|U|)-(2p-1)^2({\lambda}-1)+2p\\
\geq &(p+1/2)(|S_1|+|S_{\lambda+1}|)+(p+1/2)(|S|-|T_1\bm\cdot S_{1}|-|S_{\lambda+1}|-|U|)+2p\\
= &(p+1/2)(|S|-|T_1|-|U|)+2p\\
\geq &(p+1/2)(n/p+p-2-(4p-3))+2p\\
\geq &n  \ \ \ \ \ \ \ \ \ (\mbox{as } n\geq 6p^3-3p^2-p),
\end{align*}
yielding a contradiction.
\\

\noindent {\bf Subcase 2.2.} $|U|\geq 2p-1$.

Let $H=\langle \mbox{supp}(U), g_{\lambda+1}\rangle$. If $|H/\langle g_{\lambda+1}\rangle|\geq 2p$, then by a similar method as used in the proof of {\bf Claim}, we can find a subsequence $T_{\lambda+1}|U$ with $|T_{\lambda+1}|=2p-1$ such that $S_{\lambda+1}$ and $T_{\lambda+1}$ satisfy Properties $(1)-(3)$, yielding a contradiction to the maximality of $\lambda$. Therefore, we must have $|H/\langle g_{\lambda+1}\rangle|\leq 2p-1$. Since $\langle g_{\lambda+1}\rangle$ is a cyclic subgroup of $G$, we have $|G/\langle g_{\lambda+1}\rangle|\geq n_1\geq 2p$. Thus $H\neq G$. Therefore, $$|H|\leq n/p\mbox{ \ \ \ \ and \ \ \ \ } |S'|\leq |S_H|\leq n/p-1 \ (\mbox{as } S'| S_H \mbox{ and }S \mbox{ is regular}).$$

Suppose that $\mbox{supp}(S_i)\nsubseteq H$ for some $i\in [1, \lambda]$. Then let $S_i'=S_i\bm\cdot g_i^{[-(p-1)]}$. We have $$S=\left(\prod_{1\leq j\neq i\leq \lambda}(T_j\bm\cdot S_j)\right)\bm\cdot (T_i\bm\cdot S_i')\bm\cdot (g_i^{[p-1]}\bm\cdot S').$$
Let $U'|U$ with $|U'|=p-1$. By Corollary~\ref{coset}~(1) and Lemma~\ref{stabilizer}~(2), $|\sum_0(U'\bm\cdot g_{\lambda+1}^{[2p-1]})|\geq (|U'|+1)|\sum_0(g_{\lambda+1}^{[2p-1]})|= 2p^2$. By Lemma~\ref{stabilizer}~(1), ${\rm st}(\sum_0(S'))=\{0\}$. Note that $$\sum\nolimits_0(S')=\sum\nolimits_0(S'\bm\cdot (U'\bm\cdot g_{\lambda+1}^{[2p-1]})^{[-1]})+\sum\nolimits_0(U'\bm\cdot g_{\lambda+1}^{[2p-1]}).$$
By Lemma~\ref{Kneser} and Lemma~\ref{stabilizer}~(2), $|\sum_0(S')|\geq |\sum_0(S'\bm\cdot (U'\bm\cdot g_{\lambda+1}^{[2p-1]})^{[-1]})|+|\sum_0(U'\bm\cdot g_{\lambda+1}^{[2p-1]})|-1\geq |S'\bm\cdot (U'\bm\cdot g_{\lambda+1}^{[2p-1]})^{[-1]}|+2p^2=|S'|+2p^2-3p+2$. Therefore, by Corollary~\ref{coset}~(1), $$|\sum\nolimits_0(g_i^{[p-1]}\bm\cdot S')|\geq(|g_i^{[p-1]}|+1)|\sum\nolimits_0(S')|\geq p|\sum\nolimits_0(S')|\geq p|S'|+2p^3-3p^2+2p.$$ As in Subcase 2.1, we have $\lambda-1\leq \frac{|S|-|T_i\bm\cdot S_i|-|S'|}{4p-2}$. Since ${\rm st}(\sum_0(S))=\{0\}$, by Lemma~\ref{Kneser},
\begin{align*}
|\sum\nolimits_0(S)|=&|\sum_{1\leq j\neq i\leq \lambda}\sum\nolimits_0(T_j\bm\cdot S_j)+\sum\nolimits_0(T_i\bm\cdot S_i')+ \sum\nolimits_0(g_i^{[p-1]}\bm\cdot S')|\\
\geq &\sum_{1\leq j\neq i\leq \lambda}(|\sum\nolimits_0(T_j\bm\cdot S_j)|-1)+(|\sum\nolimits_0(T_i\bm\cdot S_i')|-1)+|\sum\nolimits_0(g_i^{[p-1]}\bm\cdot S')|\\
\geq &\sum_{1\leq j\neq i\leq \lambda}(2p|T_j\bm\cdot S_j|-(2p-1)^2)+(2p|S_i'|+2p-1)+p|S'|+2p^3-3p^2+2p\\
= &2p(|S|-|S_i\bm\cdot T_i|-|S'|)-(2p-1)^2(\lambda-1)+2p(|S_i|-p+1)\\
  &+p|S'|+2p^3-3p^2+4p-1\\
\geq &(p+1/2)(|S|-|S_i\bm\cdot T_i|-|S'|)+2p|S_i|+p|S'|+2p^3-5p^2+6p-1\\
= &(p+1/2)|S|+(p-1/2)|S_i|-(p+1/2)|T_i|-|S'|/2+2p^3-5p^2+6p-1\\
\geq &(p+1/2)(n/p+p-2)-(2p-1)-(n/p-1)/2+2p^3-5p^2+6p-1\\
\geq &n,
\end{align*}
yielding a contradiction.

Next assume that $\mbox{supp}(S_i)\subseteq H$ for every $i\in [1, \lambda]$. Let $$S''=S_1\bm\cdot \ldots \bm\cdot S_{\lambda}\bm\cdot S'.$$
Then $S''|S_H$. Let
$$n'=|H|\leq n/p.$$

 Since $S''=S_1\bm\cdot \ldots \bm\cdot S_{\lambda}\bm\cdot S'$ with $|S_i|\geq |T_i|$ for $i\in [1, \lambda]$, we have $|S''|\geq (|S|+|S'|)/2\geq ((n/p+p-2)+|U|+|S_{\lambda+1}|)/2\geq (n'+|U|+|S_{\lambda+1}|+1)/2$.

Let $T=h_1\bm\cdot h_2 \bm\cdot S_{\lambda+1}$ where $h_1\bm\cdot h_2|U$. Let $t\geq 0$ be the maximal integer such that $S''$ has a new factorization
$$S''=A_1\bm\cdot \ldots \bm\cdot A_t\bm\cdot W\bm\cdot T$$
where each $A_i$ is a $2$-zero-sum free $3$-subset of $G$, and $W$ is a subsequence of $S''$ which contains no $2$-zero-sum free $3$-subset of $G$. Then
\begin{equation}
3t+|W|+|S_{\lambda+1}|+2=|S''|\geq (n'+|U|+|S_{\lambda+1}|+1)/2 \label{S''}
\end{equation}
and $$\mbox{supp}(W)\subseteq \{u_1,-u_1,u_2,-u_2\}$$
for some distinct elements $u_1,u_2\in H$. By Lemma~\ref{A=2}, $|\sum_0(A_i)|\geq 7$ for $i\in [1,t]$.

We now show that $|W|>|U|$. Since $S''\subseteq H$ and $|H|=n'$, we must have
$$\sum\nolimits_0(S'')\subsetneq H, \mbox{\ \ \ and thus\ \ \ }|\sum\nolimits_0(S'')|\leq n'-1.$$
For otherwise, $\sum_0(S'')=H$ and thus ${\rm st}(\sum_0(S''))\supseteq H\neq \{0\}$. However, by Lemma~\ref{stabilizer}~(1), ${\rm st}(\sum_0(S''))=\{0\}$, yielding a contradiction. By Corollary~\ref{coset}~(1) and Lemma~\ref{stabilizer}~(2), $|\sum\nolimits_0(T)|\geq 3(|S_{\lambda+1}|+1)$. It follows from Lemma~\ref{Kneser} and \ref{stabilizer}~(2) that
\begin{align*}
n'-1\geq &|\sum\nolimits_0(S'')|=|\sum_{i=1}^t\sum\nolimits_0(A_i)+\sum\nolimits_0(W)+ \sum\nolimits_0(T)|\\
\geq &\sum_{i=1}^t(|\sum\nolimits_0(A_i)|-1)+(|\sum\nolimits_0(W)|-1)+|\sum\nolimits_0(T)|\\
\geq &\sum_{i=1}^t6+|W|+3(|S_{\lambda+1}|+1)\\
= &6t+|W|+3|S_{\lambda+1}|+3\\
= &2|S''|-|W|+|S_{\lambda+1}|-1       \\
\geq &n'+|U|+2|S_{\lambda+1}|-|W| \  \ \ \ \ \ \ \ (\mbox{by } \eqref{S''}).
\end{align*}
This gives $$|W|\geq |U|+2|S_{\lambda+1}|+1>|U|.$$
Note that $W|S''\bm\cdot S_{\lambda +1}^{[-1]}=S_1\bm\cdot \ldots \bm\cdot S_{\lambda}\bm\cdot U$. In view of Property (2), if $u_1$ lies in $\supp(S_j)$ for some $j\leq \lambda+1$, then $\mathsf v_{\pm u_1}(S'')=\mathsf v_{\pm u_1}(S_j)$. Therefore, $u_1\in \supp(U)$ if and only if $\mathsf v_{\pm u_1}(U)=\mathsf v_{\pm u_1}(S'')$, meaning all terms equal to $u_1$ or $-u_1$ in $S''$ occur in $U$ provided that $u_1\in \supp(U)$; and likewise for $-u_1$, $u_2$ and $-u_2$. Since $W|S''$ and $\mbox{supp}(W)\subseteq \{u_1,-u_1,u_2,-u_2\}$, we have $W|u_1^{\mathsf v_{u_1}(S'')}\bm\cdot (-u_1)^{\mathsf v_{-u_1}(S'')}\bm\cdot u_2^{\mathsf v_{u_2}(S'')}\bm\cdot (-u_2)^{\mathsf v_{-u_2}(S'')}$. If $u_1\bm\cdot u_2|W$ and $u_1\bm\cdot u_2|U$, then by what we just proved, $u_1^{\mathsf v_{u_1}(S'')}\bm\cdot (-u_1)^{\mathsf v_{-u_1}(S'')}\bm\cdot u_2^{\mathsf v_{u_2}(S'')}\bm\cdot (-u_2)^{\mathsf v_{-u_2}(S'')}|U$, so $W|U$, implying $|W|\leq |U|$, yielding a contradiction to $|W|> |U|$. Therefore, without loss of generality, we may assume that $u_1|W$ and $u_1\notin \supp(U)$. Thus $u_1\in \mbox{supp}(S_j)$ for some $j\leq \lambda+1$. Since the sequence $W$ is disjoint from $S_{\lambda+1}$ by construction, it follows that $u_1\in \supp(S_j)$ for some $j\leq \lambda$. In view of Property (2), we conclude that $-u_1\in \mbox{supp}(S_j)$ and $v\notin \langle u_1\rangle\leq \langle g_j\rangle$ for every $v\in \mbox{supp}(S_{\lambda+1})$. Write $W=W_1\bm\cdot W_2$ with
$$W_1|u_1^{[\mathsf v_{u_1}(S'')]}\bm\cdot (-u_1)^{[\mathsf v_{-u_1}(S'')]}\mbox{ \ \ and \ \ }W_2|u_2^{[\mathsf v_{u_2}(S'')]}\bm\cdot (-u_2)^{[\mathsf v_{-u_2}(S'')]}.$$
Now fix $v|S_{\lambda+1}$. Let $V=W_1\bm\cdot v$ and $T'=T\bm\cdot v^{[-1]}$. We obtain another factorization $$S''=A_1\bm\cdot \ldots \bm\cdot A_t\bm\cdot V \bm\cdot W_2 \bm\cdot T'.$$
Since $v\notin \langle u_1\rangle$, by Corollary~\ref{coset}~(1) and Lemma~\ref{stabilizer}~(2), $|\sum\nolimits_0(V)|\geq 2(|W_1|+1)$. Note that
\begin{equation}
|S''|=3t+|V|+|W_2|+|T|-1=3t+|W_1|+|W_2|+|S_{\lambda+1}|+2. \label{W_1}
\end{equation}
As before, we obtain

\begin{align*}
n'-1\geq &|\sum\nolimits_0(S'')|=|\sum_{i=1}^t\sum\nolimits_0(A_i)+\sum\nolimits_0(V)+\sum\nolimits_0(W_2)+ \sum\nolimits_0(T')|\\
\geq &\sum_{i=1}^t(|\sum\nolimits_0(A_i)|-1)+(|\sum\nolimits_0(V)|-1)+(|\sum\nolimits_0(W_2)|-1)+|\sum\nolimits_0(T')|\\
\geq &\sum_{i=1}^t6+2|W_1|+1+|W_2|+3|S_{\lambda+1}|\\
= &6t+2|W_1|+|W_2|+3|S_{\lambda+1}|+1\\
= &2|S''|-|W_2|+|S_{\lambda+1}|-3     \ \ \ \ \ \ \ (\mbox{by } \eqref{W_1})\\
\geq &n'+|U|+2|S_{\lambda+1}|-|W_2|-2 \  \ \ \ \ \ \ \ (\mbox{by } \eqref{S''}),
\end{align*}
implying $$|W_2|\geq |U|+2|S_{\lambda+1}|-1>|U|.$$
Next we consider $W_2$. Since $\supp(W_2)\subseteq\{u_2,-u_2\}$, if $u_2\in \supp(U)$, then as above, we obtain $W_2|U$, implying $|W_2|\leq |U|$, giving a contradiction. Thus $u_2\notin \supp(U)$. As above, we conclude that $\{\pm u_2\}\subseteq \mbox{supp}(S_{j'})$ for some $j'\leq \lambda$ and $z\notin \langle u_2\rangle\leq \langle g_{j'}\rangle$ for every $z\in \mbox{supp}(S_{\lambda+1})$. Since $|S_{\lambda+1}|\geq 2p-1\geq 5$, we can choose $z|S_{\lambda+1}\bm\cdot v^{[-1]}$. Define $V=W_1\bm\cdot v$, $Z=W_2\bm\cdot z$ and $T''=T\bm\cdot v^{[-1]}\bm\cdot z^{[-1]}$. Thus we obtain another factorization $$S''=A_1\bm\cdot \ldots \bm\cdot A_t\bm\cdot V \bm\cdot Z \bm\cdot T''.$$
As before, we obtain
\begin{align*}
n'-1\geq &|\sum\nolimits_0(S'')|=|\sum_{i=1}^t\sum\nolimits_0(A_i)+\sum\nolimits_0(V)+\sum\nolimits_0(Z)+ \sum\nolimits_0(T'')|\\
\geq &\sum_{i=1}^t(|\sum\nolimits_0(A_i)|-1)+(|\sum\nolimits_0(V)|-1)+(|\sum\nolimits_0(Z)|-1)+|\sum\nolimits_0(T'')|\\
\geq &\sum_{i=1}^t6+2|W_1|+1+2|W_2|+1+3(|S_{\lambda+1}|-1)\\
= &6t+2|W_1|+2|W_2|+3|S_{\lambda+1}|-1\\
= &2|S''|+|S_{\lambda+1}|-5      \ \ \ \ \ \ \ (\mbox{by } \eqref{W_1})\\
\geq &n'+|U|+2|S_{\lambda+1}|-4  \  \ \ \ \ \ \ \ (\mbox{by } \eqref{S''})\\
\geq &n' \ \ \ \ \ \ (\mbox{as } |U|\geq 2p-1\geq 5),
\end{align*}
yielding a contradiction. In all cases we have found contradictions. Thus we must have $\sum(S)=G$. This completes the proof.         \qed
\\

We end this paper with the following remark.
\begin{remark}
It was proved that $f(6)=19$ in  \cite[Theorem 1.2]{GLPS}. Thus Lemma~\ref{subsetsums} can be improved as follows: if $A\subseteq G\setminus\{0\}$ is a subset with $|A|\geq 21$, then there is a subset $B\subset A$ such that $|\sum_0(B)|\geq 3|B|+2$ with $|B|\leq 6$. Therefore, when $p=3$, we may relax the condition $|G|\geq 72p^6=52488$ to $|G|\geq 3045$, and the same proof as in Theorem~\ref{mainthm} still works.
\end{remark}

\noindent {\bf Acknowledgements}. Part of this work was carried out during a visit of the first author to Brock University as an international visiting scholar. He would like to gratefully thank the host institution for its hospitality and for providing an excellent atmosphere for research. This work was supported in part by the National Natural Science Foundation of China (No. 11701256, 11871258, 12071344), the Youth Backbone Teacher Foundation of Henan's University (No. 2019GGJS196), the China Scholarship Council (Grant No. 201908410132), and it was also supported in part by a Discovery Grant from the Natural Sciences and Engineering Research Council of Canada (Grant No. RGPIN 2017-03903).

\end{document}